\documentclass[12pt]{article}

\usepackage[a4paper,margin=2cm]{geometry}
\usepackage[english]{babel}
\usepackage{latexsym,amsmath,enumerate,graphics,enumerate,amsthm,tikz,hyperref,float}
\usepackage[mathscr]{euscript}
\usepackage[affil-it]{authblk}
\usepackage{enumerate,amsthm,dsfont,pstricks}
\usepackage{latexsym,amsmath,amssymb,amscd,wrapfig,graphicx}
\usepackage[normalem]{ulem}

\newtheorem{theorem}{Theorem}[section]
\newtheorem{lemma}[theorem]{Lemma}
\newtheorem{proposition}[theorem]{Proposition}
\newtheorem{corollary}[theorem]{Corollary}

\theoremstyle{definition}
\newtheorem{definition}[theorem]{Definition}

\theoremstyle{remark}
\newtheorem{remark}[theorem]{Remark}

\newtheorem*{theorem*}{{\bf Theorem}}

\newtheorem*{assumption*}{{\bf Assumption}}

\let\phi=\varphi
\def\eps{\varepsilon}

\def\N{\mathbb{N}}
\def\R{\mathbb{R}}
\def\C{\mathbb{C}}

\def\H{\mathbb{H}}
\def\F{\mathbb{F}}

\def\eps{\varepsilon}

\def\0{\mathbf{0}}

\def\P{\mathcal{P}}

\def\ol{\overline}

\DeclareMathOperator{\rank}{rank}

\def\Poset{{\textnormal {\bf Poset}}}
\def\Pozero{{\textnormal {\bf Poset0}}}

\newcommand{\comment}[1]{}
\newcommand{\norm}[1]{\left\Vert #1 \right\Vert}

\numberwithin{equation}{section}

\let\epsilon=\varepsilon

\makeatletter
\def\@maketitle{%
  \newpage
  \null
  \vskip 2em%
  \begin{center}%
  \let \footnote \thanks
    {\Large\bfseries \@title \par}%
    \vskip 1.5em%
    {\normalsize
      \lineskip .5em%
      \begin{tabular}[t]{c}%
        \@author
      \end{tabular}\par}%
    \vskip 1em%
    {\normalsize \@date}%
  \end{center}%
  \par
  \vskip 1.5em}
\makeatother

\begin{document}

\title{\sc \huge Order isomorphisms on order intervals of atomic JBW-algebras}

\author{Mark Roelands%
\thanks{Email: \texttt{mark.roelands@gmail.com}}}
\affil{School of Mathematics, Statistics \& Actuarial Science, University of Kent, Canterbury, CT2 7NX,
United Kingdom}

\author{Marten Wortel%
\thanks{Email: \texttt{marten.wortel@up.ac.za}}}
\affil{Department of Mathematics and Applied Mathematics, University of Pretoria, Private Bag X20 Hatfield, 0028 Pretoria, South Africa}
\maketitle
\date{}
%\date{\today}

\begin{abstract}
In this paper a full description of order isomorphisms between effect algebras of atomic JBW-algebras is given. We will derive a closed formula for the order isomorphisms on the effect algebra of type I factors by proving that the invertible part of the effect algebra of a type I factor is left invariant. This yields an order isomorphism on the whole cone, for which a characterisation exists. Furthermore, we will show that the obtained formula for the order isomorphism on the invertible part can be extended to the whole effect algebra again. As atomic JBW-algebras are direct sums of type I factors and order isomorphisms factor through the direct sum decomposition, this yields the desired description.
\end{abstract}



%\tableofcontents
\section{Introduction}

Let $H$ be a complex Hilbert space. In \cite{Semrl}, inspired by Ludwig \cite[Section~V.5]{Ludwig} and Moln\'ar \cite[Corollary~4]{Molnar2003}, \v{S}emrl investigated the order isomorphisms of the effect algebra 
$$
[0, I] := \{T \in B(H)_{sa} \colon 0 \leq T \leq I\},
$$
which is an important object in quantum mechanics. For more details and an extensive historical perspective we direct the reader to the excellent introduction of \v{S}emrl's paper \cite{Semrl}. The previous results of Ludwig and Moln\'ar characterised the order isomorphisms of the effect algebra that satisfy some additional assumptions, but \v{S}emrl was the first to provide a characterisation in full generality.

In his paper, \v{S}emrl proved and made use of the fact that each order isomorphism of the effect algebra leaves the invertible part invariant. Drnov\v{s}ek made clever use of this fact in \cite{D} by noting that the inverse yields an order anti-isomorphism between the invertible part of the effect algebra and the translated positive cone. Since the order isomorphisms of the positive cone are known (\cite[Theorem~1]{Molnar2001}), Drnov\v{s}ek was able to characterise the order isomorphisms of the effect algebra, greatly simplifying the proof of \v{S}emrl's result.

Note that $B(H)$ is an atomic von Neumann algebra, in fact a factor, and its self-adjoint part is therefore an atomic JBW-algebra (the reader more familiar with the theory of $C$*-algebras should think of the self-adjoint part of a von Neumann algebra whenever JBW-algebras are mentioned). The goal of this paper is to extend the above results on $B(H)$ to arbitrary atomic JBW-algebras. 

Our approach is based on Drnov\v{s}ek's idea of using the order anti-isomorphism of the invertible part of the effect algebra with the translated positive cone. For this we need that the invertible part is invariant, and we show that this is true if the atomic JBW-algebra is a factor. General atomic JBW-algebras are direct sums (of arbitrary size) of factors, and if there are infinitely many summands then the invertible part need no longer be invariant, but we also show that order decompositions correspond  to JBW-algebra direct sum decompositions, a result that allows us to treat each factor separately. We then apply a recent result by Roelands and van Imhoff \cite[Theorem~3.8]{IR} characterising order isomorphisms between the positive cones in atomic JBW-algebras to obtain a full description of order isomorphisms between their effect algebras.

Another major contribution in this research area was made by Mori in \cite{Mori}, who mostly considered order isomorphisms between the effect algebra of (not necessarily atomic) von Neumann algebras without type $I_1$ direct summand. He showed that the image of $\frac{1}{2}I$ is always locally measurably invertible, a concept from the theory of non-commutative integration. Furthermore, he characterised the order isomorphisms in the case where the image of $\frac{1}{2}I$ is invertible. 
\\ \\
We will now briefly outline the structure of our paper. Section~\ref{s:prelim} is our preliminary section where we cover the basics of JB(W)-algebras and order products. In Section~\ref{s:order_structure} we investigate the order structure of the effect algebra. The main result of this section is that the JBW-algebra decompositions are in bijection with the order decompositions of the space. This is used in Section~\ref{s:order_isomorphisms}, where we first show that the invertible part of the effect algebra is always invariant if and only if the centre is finite-dimensional, and use this to obtain a full characterisation of order isomorphisms of the effect algebra in atomic JBW-algebras, Theorem~\ref{t:char order isoms finite sums}.

\section{Preliminaries}\label{s:prelim}

\subsection{JB-algebras}

A \emph{Jordan algebra} $(A, \circ)$ is a commutative, not necessarily associative algebra such that
\[
x \circ (y \circ x^2) = (x \circ y) \circ x^2 \mbox{\quad  for all }x,y \in A.
\]
A \emph{JB-algebra} $A$ is a normed, complete real Jordan algebra satisfying,
\begin{align*}
\norm{x \circ y} &\leq \norm{x}\norm{y}, \\
\norm{x^2} &= \norm{x}^2, \\
\norm{x^2} &\leq \norm{x^2 + y^2}
\end{align*}
for all $x,y \in A$. The canonical example of a JB-algebra is the set of self-adjoint elements of a $C^*$-algebra equipped with the Jordan product $x \circ y := \frac{1}{2}(xy + yx)$. By the Gelfand-Naimark theorem, this JB-algebra is a norm closed Jordan subalgebra of the self-adjoint bounded operators on a complex Hilbert space. Such a JB-algebra is called a \emph{JC-algebra}.

The elements $x,y \in A$ are said to \emph{operator commute} if $x \circ (y \circ z) = y \circ (x \circ z)$ for all $z \in A$. In a JC-algebra, two elements operator commute if and only if they commute for the $C^*$-multiplication by \cite[Proposition~1.49]{AS}. An element $x\in A$ is said to be \emph{central} if it operator commutes with all elements of $A$. The \emph{centre} of $A$, denoted by $Z(A)$, consists of all elements that operator commute with all elements of $A$, and it is an associative subalgebra of $A$. Every associative unital JB-algebra is isometrically isomorphic to $C(K)$ for some compact Hausdorff space $K$, see \cite[Theorem~3.2.2]{OS}.

The set of invertible elements of a unital JB-algebra A is denoted by $\mathrm{Inv}(A)$. The \emph{spectrum} of $x\in  A$, which is denoted by $\sigma(x)$, is
defined to be the set of $\lambda\in\mathbb{R}$ such that $\lambda e-x$ is not invertible in JB$(x, e)$, the JB-algebra
generated by $x$ and $e$, see \cite[3.2.3]{OS}. Furthermore, there is a continuous functional calculus, that is, JB$(x, e)$ is isometrically isomorphic to $C(\sigma(x))$ as a JB-algebra. The cone of elements with non-negative spectrum is denoted by $A_+$, and equals the set of squares by the functional calculus, and its interior $A_+^\circ$ consists of all elements with strictly positive spectrum.
The cone $A_+$ induces a partial ordering $\leq$ on $A$ by writing $x\le y$ if $y-x\in A_+$. 

An \emph{order interval} in $A$ is of the form $[x,y]$ for $x,y\in A$ with $x\le y$. For order intervals $[x,y]$ and $[u,v]$ an \emph{order isomorphism} $f\colon [x,y]\to[u,v]$ is an order preserving bijection with order preserving inverse, that is, $x\le y$ if and only if $f(x)\le f(y)$. Since we have $[x,y]=x+[0,y-x]$ and $[x,y]$ is order isomorphic to $[0,y-x]$ via translation by $x$, it follows that $f\colon [x,y]\to [u,v]$ is an order isomorphism if and only if $\hat{f}\colon [0,y-x]\to [0,v-u]$ is an order isomorphism. Hence the order isomorphisms between order intervals in JB-algebras are completely determined by order isomorphisms between order intervals the form $[0,x]$.  

The \emph{Jordan triple product} $\{ \cdot, \cdot, \cdot \}$ is defined as
\[ \{x,y,z \} := (x \circ y) \circ z + (z \circ y) \circ x - (x \circ z) \circ y, \]
for $x,y,z \in A$. The linear map $U_x \colon A \to A$ defined by $U_x y := \{x,y,x\}$ will play an important role and is called the \emph{quadratic representation} of $x$.

By the Shirshov-Cohn theorem for JB-algebras \cite[Theorem~7.2.5]{OS}, the unital JB-algebra
generated by two elements is a JC-algebra, which shows all but the fourth of the following identities
for JB-algebras, since $U_xy = xyx$ in JC-algebras (for the rest of the paper, the operator-algebraic
reader is encouraged to think of this equality whenever the quadratic representation appears).

\begin{align}\label{eq:quad properties}
&U_yx\in A_+ & (\forall x\in A_+,\ y\in A)\nonumber\\ &U_x^{-1}=U_{x^{-1}} &(\forall x\in\mathrm{Inv}(A))\nonumber\\ &(U_yx)^{-1}=U_{y^{-1}}x^{-1} &(\forall x,y\in\mathrm{Inv}(A))\\ 
&U_yU_xU_y=U_{U_yx} & (\forall x,y\in A)\nonumber\\& U_ye=y^2 & (\forall y\in A)\nonumber
\end{align}

A proof of the fourth identity can be found in \cite[2.4.18]{OS}, as well as proofs of the other identities. Moreover, if $f,g \colon \sigma(y) \to\mathbb{R}$ are continuous functions, then by passing to the subalgebra generated by $y$ it follows that
\begin{align}\label{e:op com quad rep}
U_{f(y)}U_{g(y)} = U_{[fg](y)} \quad \mbox{and} \quad U_{f(y)} g(y) = [f^2 g](y).
\end{align}

\comment{The following lemma also follows from the same idea combined with the continuous functional calculus.

\begin{lemma}\label{l:op com quad rep}
Let $A$ be a unital JB-algebra and let $y\in A$. If $f,g \colon \sigma(y) \to\mathbb{R}$ are continuous functions, then 
\begin{enumerate}
\item $U_{f(y)}U_{g(y)} = U_{[fg](y)}$;
\item $U_{f(y)} g(y) = [f^2 g](y)$.
\end{enumerate}

\end{lemma}

\begin{proof}
$(i)$: For any $x\in A$ the JB-algebra $B$ generated by $x$, $y$, and $e$ is a JC-algebra by the Shirshov-Cohn theorem for JB-algebras and it contains JB$(y,e)\cong C(\sigma(y))$. Hence in $B$ we have
\[
U_{f(y)}U_{g(y)}x = f(y)g(y)xg(y)f(y) = [fg](y) x [fg](y) = U_{[fg](y)} x
\]
which shows that the required identity holds in $A$ as $x$ was chosen to be arbitrary.

$(ii)$: Since $f(y), g(y) \in \text{JB}(y,e) \cong C(\sigma(y))$, it follows that $U_{f(y)} g(y) = f(y) g(y) f(y) = [f^2g](y)$.
\end{proof}

}

A \emph{JBW-algebra} $M$ is the Jordan analogue of a von Neumann algebra: it is a JB-algebra with unit $e$ which is monotone complete and has a separating set of normal states, or equivalently, a JB-algebra that is a dual space \cite[Theorem~4.4.16]{OS}. A positive functional $\varphi\in M^*$ is called a \emph{state} if $\varphi(e)=1$, and it is said to be {\it normal} if for any bounded increasing net $(x_i)_{i\in I}$ with supremum $x$ we have $\varphi(x_i)\rightarrow\varphi(x)$. The set of normal states on $M$ is called the {\it normal state space} of $M$. The topology on $M$ defined by the duality of $M$ and the normal state space of $M$ is called the {\em $\sigma$-weak topology}. That is, we say a net $(x_i)_{i\in I}$ converges $\sigma$-weakly to $x$ if $\varphi(x_i)\to \varphi(x)$ for all normal states $\varphi$ on $M$. Since the linear span of the normal state space is the predual of $M$, the $\sigma$-weak topology equals the weak* topology. The \emph{$\sigma$-strong topology} on $M$ is generated by the semi-norms $x\mapsto\varphi(x^2)^{1/2}$ for all $x\in M$ where $\varphi$ ranges over the normal states on $M$. By \cite[Proposition~2.4]{AS} the Jordan multiplication on $M$ is separately $\sigma$-weakly continuous in each variable and jointly $\sigma$-strongly continuous on bounded sets, and for any $x\in M$, the corresponding quadratic representation $U_x$ is $\sigma$-weakly continuous.

We will now collect some results on the functional calculus that are standard in von Neumann algebras but for which there does not seem to be a good reference in the theory of JBW-algebras. Let $M$ be a JBW-algebra and let $x \in M$. Then $W(x,e)$, the JBW-algebra generated by $x$ and $e$, is an associative JBW-algebra and hence isomorphic to a monotone complete $C(K)$-space and so its complexification $\mathcal{M}$ is a von Neumann algebra also generated by $x$ and $e$. A standard result in von Neumann algebras, which can for example easily be derived from \cite[Theorem~IX.2.3]{conway}, now yields the bounded functional calculus $f \mapsto f(x)$ from the bounded Borel functions on $\sigma(x)$ to $\mathcal{M}$. It follows from this theorem and the dominated convergence theorem that if $f_n$ is a uniformly bounded sequence converging pointwise to $f$, then $f_n(x)$ converges to $f(x)$ in the weak operator topology and hence $\sigma$-weakly by \cite[Lemma~II.2.5(i)]{Takesaki}. Restricting this functional calculus to real-valued functions now yields the bounded functional calculus into $W(x,e)$ with a similar property: if $f_n$ is a uniformly bounded sequence converging pointwise to $f$, then $f_n(x)$ converges $\sigma$-weakly to $f(x)$. The theorem below summarises the above discussion.

\begin{theorem}[Bounded functional calculus]\label{t:bdd_func_calc}
Let $M$ be a JBW-algebra and let $x \in M$. Then the map $f \mapsto f(x)$ is a unital algebra homomorphism from the bounded Borel functions on $\sigma(x)$ onto $W(x,e)$ with the property that if $f_n$ is a uniformly bounded sequence of functions converging pointwise to $f$, then $f_n(x)$ converges $\sigma$-weakly to $f(x)$.
\end{theorem}

%By \cite[IX.8.10]{conway}, there is a von Neumann algebra isomorphism $f \mapsto f(x)$ from $L^\infty(\sigma(x), \mu)$ into $\mathcal{M}$ where $\mu$ is a scalar-valued spectral measure for the self-adjoint operator $x$, that extends the continuous functional calculus. Restricting this isomorphism to the real part yields a JBW-algebra isomorphism $L^\infty(\sigma(x), \mu)_{\R} \cong W(x,e)$, the bounded functional calculus for $x$. 

An element $p$ in a JBW-algebra $M$ is a {\it projection} if $p^2=p$. For a projection $p\in M$ the {\it orthogonal complement}, $e-p$, will be denoted by $p^\perp$ and a projection $q$ is {\it orthogonal} to $p$ precisely when $p\circ q=0$, which is equivalent to $q\le p^\perp$ by \cite[Proposition~2.18]{AS}. The collection of projections $\mathcal{P}(M)$ forms a complete orthomodular lattice by \cite[Proposition~2.25]{AS} and is referred to as the \emph{projection lattice} of $M$. We remark that every set of projections considered as a subset of $\P(M)$ has a supremum in $\P(M)$, but it need not have a supremum when considered as a subset of $M$. For a positive element $x\in M$, the smallest projection $p$ such that $U_px=x$ is called the \emph{range projection} of $x$ and is denoted by $r(x)$. See \cite[Proposition~2.13]{AS}. By the proof of \cite[Lemma~4.2.6]{OS} and the bounded functional calculus, $r(x) = \mathbf{1}_{(0, \infty)}(x)$.  Furthermore, if $p$ is a projection and $0 \leq x \leq p$, then $r(x) \leq r(p) = p$ by \cite[Proposition~2.15~(2.7)]{AS}, and $U_p x = x$ by \cite[Proposition~2.15~(2.8)]{AS}, or equivalently, $x \in U_p M$.

Any central projection $p$ decomposes the JBW-algebra $M$ as a direct sum of JBW-subalgebras such that $M=U_pM\oplus U_{p^\perp} M$, see \cite[Proposition~2.41]{AS}. A minimal non-zero projection is called an \emph{atom}, and it follows from \cite[Proposition~2.32, Lemma~3.29]{AS} that a non-zero projection $p$ is an atom if and only if the order interval $[0,p]$ is totally ordered. A JBW-algebra in which every non-zero projection dominates an atom is called \emph{atomic}. Every JBW-algebra decomposes as a direct sum of type I, type II, and type III JBW-algebras, and a JBW-algebra with trivial centre is called a \emph{factor}. In this paper we will predominantly be working with atomic JBW-algebras, which by \cite[Proposition~3.45]{AS} are a direct sum of type I factors. The type I JBW-factors have been classified and are up to isomorphism either a spin factor, or $\mathrm{Mat}_3(\mathbb{O})_{sa}$, the self-adjoint $3\times 3$ matrices with octonionic entries, or of the form $B(H)_{sa}$, the bounded self-adjoint operators on $H$ where $H$ is a real, complex, or quaternionic Hilbert space of dimension $d \geq 3$.  See \cite[Theorem~3.39]{AS}.

A \emph{Jordan isomorphism} between JB-algebras is a bijection that preserves the Jordan structure. In \cite[Section~2]{jbisometries}, the Jordan isomorphisms of the type I JBW-factors (except $\mathrm{Mat}_3(\mathbb{O})_{sa}$) were characterised: on a spin factor they are induced by unitaries on the Hilbert space, and on $B(H)_{sa}$ with a Hilbert space $H$ of dimension $d \geq 3$ over $\F = \R$, $\C$, or $\H$, every Jordan isomorphism $J$ is of the form $Jx = uxu^{-1}$, where $u$ is a surjective real-linear isometry satisfying $u(\lambda h) = \tau(\lambda) u(h)$ ($\lambda \in \F$, $h \in H$) for some automorphism $\tau$ of $\F$. So $u$ is linear or conjugate linear if $H$ is complex. The group of Jordan isomorphisms of $\mathrm{Mat}_3(\mathbb{O})_{sa}$ is the exceptional Lie group $F_4$ (\cite{CS}).

For an atomic JBW-algebra $M$, we define the \emph{rank} of $M$ to be the cardinality of a maximal orthogonal collection of atoms. For example, $\rank(\ell^\infty(S)) = |S|$ and $\rank(B(H)_{sa}) = \dim(H)$ for any set $S$ and for any real, complex or quaternionic Hilbert space $H$.

\subsection{Order products}

In the sequel we will consider the category \Poset{} of partially ordered sets with monotone maps as morphisms, and the category \Pozero{} of posets with a distinguished element $0$ (not necessarily the least element) and monotone maps that preserve $0$ as morphisms. As usual, we will use the notation $S \cong T$ to denote partially ordered sets $S$ and $T$ that are order isomorphic. Given a collection of posets $\{S_i\}_{i \in I}$ in \Poset{} or \Pozero{}, we define the order product $\prod_{i \in I} S_i$ to be the cartesian product of $\{S_i\}_{i \in I}$ equipped with the product ordering. It is an obvious exercise to show that the order product equipped with the canonical projections is the categorical product in \Poset{} and \Pozero{}, and that suprema and infima exist in $\prod_{i \in I} S_i$ if and only if they exist coordinatewise, in which case they are given by the coordinatewise operations. 

Given $\{S_i\}_{i \in I}$  in $\Pozero{}$, if $s_i \in S_i$ for some $i$, then we denote by $\ol{s_i}$ the element in $\prod_{i \in I} S_i$ which equals $s_i$ in the $i$-th coordinate and $0$ elsewhere. 

\begin{definition}\label{d:internal_order_product}
Let $S \in \Pozero$ and $\{S_i\}_{i \in I}$ be subobjects, i.e., subposets with $0$. Then we define $S$ to be the \emph{internal order product} of $\{S_i\}_{i \in I}$ if there is an order isomorphism between $S$ and $\prod_{i \in I} S_i$, denoted by $s \sim (s_i)_{i \in I}$, such that $s_i \sim \ol{s_i}$ for each $i \in I$ and $s_i \in S_i$.
\end{definition}

\begin{remark}\label{r:sup_in_internal_order_product}
Suppose that $S \in \Pozero{}$ is the internal order product of $\{S_i\}_{i \in I}$ and $0$ is the least element of $S$. If $s_i \in S_i$ for each $i \in I$, then
$$
(s_i)_{i \in I} = \bigvee_{i \in I} \ol{s_i} \sim \bigvee_{i \in I} s_i.
$$
Hence $(s_i)_{i \in I} \in \prod_{i \in I} S_i$ is identified with $\bigvee_{i \in I} s_i \in S$.
\end{remark}

\section{The order structure of $[0,e]$}\label{s:order_structure}

In this section we study the relation between the order structure of the effect algebra $[0,e]$ of a JBW-algebra $M$ and that of its projection lattice $\mathcal{P}(M)$, as well as the connection between decompositions of the interval and algebraic direct summands of $M$.

If $T$ is a positive injective operator on a Hilbert space, then \cite[Theorem~2.8]{Semrl} shows that $[0,I]$ is order isomorphic to $[0,T]$. For self-adjoint operators being injective is equivalent to having dense range, which in turn is equivalent to having range projection $I$. Our next proposition generalises this theorem to JBW-algebras. In case of von Neumann algebras Douglas' lemma yields a simpler proof, but the absence of a Hilbert space makes our proof more elaborate.

\begin{proposition}\label{p:order isomorphic intervals}
Let $M$ be a JBW-algebra and let $x\in M_+$ be non-zero with range projection $r(x)$. Then the quadratic representation $U_{x^{1/2}}\colon [0,r(x)]\to [0,x]$ is an order isomorphism.
\end{proposition}

\begin{proof}
All limits in this proof are $\sigma$-weak limits. For $n\in\N$ define continuous functions $f_n$ by
\[
f_n(t):=
\begin{cases}
t^{-{1/2}} & \mbox{if }t\in [\frac{1}{n},\infty) \\
n^{3/2} t & \mbox{if }t\in [0,\frac{1}{n})
\end{cases} .
\]
We also define 
\[ 
g_n(t) := tf_n(t)^2 = \begin{cases}
1 & \mbox{if }t\in [\frac{1}{n},\infty) \\
n^3 t^3 & \mbox{if }t\in [0,\frac{1}{n})  
\end{cases}, \quad
h_n(t) := t^{1/2} f_n(t) = \begin{cases}
1 & \mbox{if }t\in [\frac{1}{n},\infty) \\
n^{3/2} t^{3/2} & \mbox{if }t\in [0,\frac{1}{n})  
\end{cases}.
\]
Since $g_n$ and $h_n$ are uniformly bounded and converge pointwise to $\mathbf{1}_{(0, \infty)}$, the bounded functional calculus (Theorem~\ref{t:bdd_func_calc}) yields that $g_n(x), h_n(x) \to \mathbf{1}_{(0, \infty)}(x) = r(x)$.

%Since $\sup_n g_n = \sup_n h_n = \mathbf{1}_{(0, \infty)}$, the bounded functional calculus yields $f_n(x) \uparrow \mathbf{1}_{(0, \infty)}(x) = r(x)$ and $f_n(x) \uparrow \mathbf{1}_{(0, \infty)}(x) = r(x)$, so $g_n(x), h_n(x) \to r(x)$ by Lemma~\ref{l:sup_w*}. 

We claim that the map $V \colon [0,x] \to [0, r(x)]$ defined by  $Vy := \lim_n U_{f_n(x)}y$ is the inverse of $U_{x^{1/2}}$. To show that $V$ is well defined we have to prove that $U_{f_n(x)}y$ actually converges. So let $0 \leq y \leq x$, then by \eqref{e:op com quad rep}, $U_{f_n(x)} y \leq U_{f_n(x)} x = g_n(x) \leq r(x)$, and as $[0,r(x)]$ is $\sigma$-weakly compact, there is a convergent subnet such that $ \lim_i U_{f_i(x)}y = z \in [0,r(x)]$ (we have not yet shown that $U_{f_n(x)}y$ converges). It follows that, using \eqref{e:op com quad rep} in the third equality and \cite[p.~40,~(2.2)]{AS} in the fourth equality,
\begin{equation}\label{e:inverse}
U_{x^{1/2}} z = U_{x^{1/2}} \lim_i U_{f_i(x)}y = \lim_i U_{x^{1/2}} U_{f_i(x)}y = \lim_i U_{h_i(x)} y = U_{r(x)} y = y.
\end{equation}
But
\[ 
U_{f_n(x)} y = U_{f_n(x)} U_{x^{1/2}} z  = U_{h_n(x)} z \to U_{r(x)} z = z,
\]
so the limit exists and therefore $V$ is well defined, $z = Vy$ and so \eqref{e:inverse} implies that $U_{x^{1/2}} Vy = y$. Similarly, if $z \in [0, r(x)]$, then
\[ 
V U_{x^{1/2}} z = \lim_n U_{f_n(x)} U_{x^{1/2}} z = U_{h_n(x)} z \to U_{r(x)} z = z
\]
showing that $V$ is the inverse of $U_{x^{1/2}}$. Since $U_{x^{1/2}}$ is positive it is order preserving. To show that $V$ is order preserving, let $y,z \in [0,x]$ be such that $y \leq z$, then $z-y \in [0,x]$ and since limits of positive elements are positive,
\[
Vz - Vy =  \lim_n U_{f_n(x)} z -  \lim_n U_{f_n(x)} y = \lim_n U_{f_n(x)} (z-y) \geq 0.
\]
Hence $U_{x^{1/2}}$ is an order isomorphism.
\end{proof}

It follows that any order interval in a JBW-algebra is order isomorphic to the order interval $[0,e]$ of some other JBW-algebra, so in the sequel we will focus on the order structure of the order intervals $[0,e]$ of JBW-algebras.

In a JBW-algebra $M$, the partially ordered set of projections $\P(M)$ forms a lattice, even though projections generally do not have a supremum in $M$; for example, $B(H)$ is an anti-lattice, meaning that any two non comparable elements have no supremum in $B(H)$. However, we proceed to show that in the effect algebra $[0,e]$ any set of projections also has a supremum. If $x \in [0,e]$, then $x^\perp$ denotes the element $e-x$; clearly $x \mapsto x^\perp$ is an order anti-isomorphism of $[0,e]$ that coincides with the usual orthogonal complement on $\P(M)$.

\begin{lemma}\label{l:proj_laticce}
Let $M$ be a JBW-algebra. In $[0,e]$, the supremum and infumum of every set of projections exists and coincides with the supremum and infimum in $\P(M)$.
\end{lemma}

\begin{proof}
Since $x \mapsto x^\perp$ is an order anti-isomorphism that leaves $\P(M)$ invariant, it suffices to show the existence of either the supremum or the infimum. We start by showing the existence of finite infima, so let $p,q \in \P(M)$. With $p \wedge q$ denoting the infimum in $\P(M)$, clearly $p \wedge q \leq p,q$. Towards showing that $p \wedge q$ is the greatest lower bound, let $0 \leq x \leq p,q$. Then $x \in U_p M$ and $x \in U_q M$ (here the fact that $x \geq 0$ is crucial), so $x \in U_p M \cap U_q M$ which is a JBW-subalgebra and hence, since $\norm{x} \leq 1$, $x$ is dominated by the identity of this JBW-subalgebra which is a projection $r \in \P(M)$. Since $r \in U_p M$ and $p$ is the largest projection in $U_p M$ it follows that $r \leq p$ and similarly $r \leq q$. Hence $r \leq p \wedge q$ (actually $r = p \wedge q$ by \cite[Proposition~2.32]{AS}), and so $x \leq r \leq p \wedge q$, showing that $p \wedge q$ is the greatest lower bound of $p$ and $q$ in $[0,e]$.

To show existence of arbitrary suprema of projections we employ the standard machinery. If $\{p_i\colon i\in I\}$ is an arbitrary collection in $\P(M)$, then for every finite $F \subseteq I$ we let $p_F := \sup_{i \in F} p_i$, where the supremum is taken in $\mathcal{P}(M)$ and hence coincides with the supremum taken in $[0,e]$. The net $(p_F)_{F \in \mathcal{F}}$ indexed by the finite subsets $F \subseteq I$ is an increasing bounded net so it converges $\sigma$-strongly to $p$ which is its supremum in $\P(M)$ by \cite[Proposition~2.25]{AS}. We claim that $p$ is the supremum of $\{p_i\colon i\in I\}$ in $[0,e]$. Clearly, $p$ is an upper bound of $\{p_i\colon i\in I\}$. Suppose $p_i \leq x \leq e$ for all $i \in I$, and let $\phi$ be a normal state on $M$. Then $\phi(p_i) \leq \phi(x)$, and since the net $(p_F)_{F \in \mathcal{F}}$ also converges $\sigma$-weakly to $p$ by \cite[Proposition~2.5(i)]{AS}, it follows that $\phi(p) \leq \phi(x)$ by taking the $\sigma$-weak limit. Thus $p \leq x$ by \cite[Corollary~2.17]{AS}, and so $p$ is the least upper bound of $\{p_i\colon i\in I\}$.
\end{proof}

Next, we show that in atomic JBW-algebras internal order decompositions are related to algebra direct sum decompositions.    

\begin{proposition}\label{p:interval decomposition}
Let $M$ be an atomic JBW-algebra. Then the following are equivalent:
\begin{enumerate}
\item $M = M_1 \oplus M_2$ as a direct sum of JBW-subalgebras;
\item $M \cong S \times T$ as an internal order product;
\item $M_+ \cong C_1 \times C_2$ as an internal order product;
\item $[0,e] \cong [0,x] \times [0,y]$ as an internal order product for some $x,y \in M$;
\item $[0,e] \cong [0,z] \times [0,z^\perp]$ as an internal order product for some central projection $z \in M$.
\end{enumerate}
Moreover, all these types of decompositions are in bijection with each other.
% If $[0,e]=[0,x]\times[0,y]$ as a partially ordered set, then $x$ and %$y$ are orthogonally complemented central projections.
%\begin{itemize}
%\item[$(i)$] If $M=M_1\oplus M_2$ as a JBW-algebra direct sum, then there is a central projection $z$ such that $M_+=U_zM_+\times U_{z^\perp}M_+$ and $[0,e]=[0,z]\times[0,z^\perp]$.
%\item[$(i)$] If $[0,e]=[0,x]\times[0,y]$ as a partially ordered set, then $x$ and $y$ are orthogonally complemented central projections.
%\item[$(ii)$] If $M_+=C_1\times C_2$ as a partially ordered set, then there is a central projection $z$ such that $U_zM_+=C_1$ and $U_{z^\perp}M_+=C_2$.
%\end{itemize}
\end{proposition}
 
\begin{proof} 
$(i) \Rightarrow (ii)$: Follows from the order on an algebra direct sum being defined coordinatewise. \\
$(ii) \Rightarrow (iii)$: Take $C_1 := \{ s \in S \colon s \geq 0\}$ and $C_2 := \{t \in T \colon t \geq 0\}$. \\
$(iii) \Rightarrow (iv)$: Let $e \sim (x,y)$ (see Definition~\ref{d:internal_order_product}) for some $x \in C_1$ and $y \in C_2$. Then 
$$
[0,e] \sim [(0,0),(x,y)] = [0,x] \times [0,y].
$$ 
$(iv) \Rightarrow (v)$: Suppose $[0,e]=[0,x]\times[0,y]$ for some $x,y\in M$. We first show that $x$ is a projection. If $u \sim (a,b) \in [0,x] \times [0,y]$ is an atom and $a$ and $b$ are non-zero, then $(a,0),(0,b) \in [0,u]$ are incomparable, contradicting the fact that $u$ is an atom, and so either $u \leq x$ or $u \leq y$. Since $0\le x\le e$, it follows that $\sigma(x)\subseteq[0,1]$. Suppose that there is a $\lambda\in\sigma(x)$ such that $0<\lambda<1$. Let $\eps > 0$ be such that $0 < \lambda-\eps < \lambda+\eps < 1$. Let $p := \mathbf{1}_{(\lambda - \eps, \lambda + \eps)}(x)$, which is non-zero by the spectral theorem (\cite[Theorem~IX.2.2(b)]{conway}), and let $v \leq p$ be an atom. Then $0 < (\lambda - \eps)v \leq (\lambda - \eps)p < x$, so $(\lambda - \eps)v$ is a common lower bound of both $v$ and $x$, therefore $v \not\leq y$ and so $v \leq x$. But then
\[
v = v \wedge x \leq v \wedge ((\lambda + \eps)p + p^\perp) = (\lambda + \eps)v < v,
\]
which is a contradiction. Hence $\sigma(x)\subseteq\{0,1\}$ and it follows that $x$ must be a projection. Similarly, we find that $y$ is a projection. Next, we will show that $x$ and $y$ are orthogonally complemented.
In $[0,x] \times [0,y]$, the only element $z$ such that $x \wedge z = 0$ and $x \vee z = e$ is $y$. By Lemma~\ref{l:proj_laticce}, $x^\perp$ satisfies these properties, so $y = x^\perp$.
%Suppose $z\in[0,e]$ is such that $z\le x,x^\perp$. Then $z\in U_xM\cap U_{x^\perp}M$ by \cite[Lemma~1.39]{AS}, but since we have $U_xM\cap U_{x^\perp}M=\{0\}$ by \cite[Proposition~1.38]{AS} and \cite[Proposition~2.9]{AS}, it follows that $x\land x^\perp=0$ in $[0,e]$. This implies that $x^\perp\le y$. Equivalently $-x$ and $-x^\perp$ have least upper bound $0$ in $[-e,0]$ and so 
%\[
%e=e+(-x)\lor(-x^\perp)=(e-x)\lor(e-x^\perp)=x^\perp\lor x
%\]
%in $[0,e]$. We conclude that $x\lor y=e=x \lor x^\perp$, so $y=x^\perp$.

It remains to show that $x$ is a central projection, and for this it suffices to show that $x$ operator commutes with any projection. If $u$ is an atom, then $[0,u]$ is totally ordered, hence $u \leq x$ or $u \leq y = x^\perp$. In both cases $x$ operator commutes with $u$. A standard argument now shows that $x$ operator commutes with any projection, which we will give for convenience of the reader. To that end, let $p\in M$ be a projection. A standard application of Zorn's lemma shows that there is a maximal set of pairwise orthogonal atoms $\{p_i\colon i\in I\}$ that are dominated by $p$. The net $(p_F)_{F\in\mathcal{F}}$ directed by the finite subsets $\mathcal{F}\subseteq I$ with $p_F:=\sum_{p_i\in F}p_i$ for $F\in\mathcal{F}$ is increasing and converges $\sigma$-strongly to a projection $q=\sup_F p_F\le p$ by \cite[Proposition~2.25]{AS}. Suppose that there is an atom $r\le p-q$. Then $r$ and $q$ are orthogonal and so $r$ is orthogonal to all atoms $p_i$, contradicting the maximality of the set $\{p_i\colon i\in I\}$. Hence $q=p$. By separate $\sigma$-strong continuity of the multiplication, it follows that $x$ operator commutes with $p$, proving that $x$ is a central projection. 

$(v) \Rightarrow (i)$: Central projections correspond to algebra direct sum decompositions.

Regarding the bijection between the different decompositions, we show that if we start with any of the types and pass through the chain of implications, we return to the original decomposition. Direct verification shows that the chain $(i) \Rightarrow (ii) \Rightarrow (iii) \Rightarrow (iv) \Rightarrow (v) \Rightarrow (i)$ leads back to the original decomposition $M \cong M_1 \oplus M_2$, and similarly for the case where we start with $(v)$. If we start with a decomposition as in $(iv)$, so $[0,e] \cong [0,x] \times [0,y]$, then our above proof shows that $x$ is a central projection with orthogonal complement $y$, so these decompositions are in direct bijection with the decompositions in $(v)$.

Suppose we start with a decomposition of type $(ii)$, so $M \cong S \times T$ as an internal order product. Going through our entire proof of $(ii) \Rightarrow (v)$, we show that in this case the order interval $[0,e]$ can be decomposed as $[0, z] \times [0, z^\perp]$ with $z$ a central projection in $S$ and $z^\perp \in T$. In a similar way, for $n \geq 2$, using elements with spectrum $\{0,n\}$ ($n$ times a projection) instead of projections, it can be shown that the order interval $[0,ne]$ can be decomposed as $[0, nz_n] \times [0, nz_n^\perp]$ with $z_n$ a central projection in $S$ and $z_n^\perp \in T$. Since $e \sim (z, z^\perp) \in [0,ne]$, so $(0,0) \leq (z, z^\perp) \leq (nz_n, nz_n^\perp)$, it follows that $z$ is a projection in $U_{z_n} M$ and therefore $z \leq z_n$. Similarly, $z^\perp \leq z_n^\perp$, and so $z = z_n$. Hence $[0, ne] \cong [0, nz] \times [0, n z^\perp]$ for all $n \geq 1$.

For $n \geq 1$, we now consider $[-ne, ne]$. Adding and subtracting $ne \sim (nz, nz^\perp)$ is an order isomorphism, so it preserves order products, and therefore 
\[
[-ne,ne] \cong [0, 2ne] \cong [0, 2nz] \times [0, 2nz^\perp] \cong [-nz, nz] \times [-nz^\perp, nz^\perp] \subseteq S \times T.
\]
Now if $x \in M$, then $x \in [-ne, ne]$ for some $n \geq 1$, and so $x \sim (a,b)$ with $a \in [-nz, nz] \subseteq S$ and $b \in [-nz^\perp, nz^\perp] \subseteq T$. This shows that $S = \cup_{n \geq 1} [-nz,nz]$ and $T = \cup_{n \geq 1} [-nz^\perp,nz^\perp]$, and so $S = U_z M$, $T = U_{z^\perp} M$, and $M \cong U_z M \times U_{z^\perp} M$ is an internal order product, which is precisely what we get if we follow the chain $ (ii) \Rightarrow (iii) \Rightarrow (iv) \Rightarrow (v) \Rightarrow (i) \Rightarrow (ii)$.

If we start with a decomposition of type $(iii)$, then we just use the above argument without the need to consider negative elements. This shows that all types of decompositions are in bijection with each other.
%Since any atom in $M$ is either dominated by $x$ or orthogonal to $x$, it follows that $x$ operator commutes with all atoms in $M$ by \cite[Proposition~2.18]{AS} and \cite[Proposition~2.26]{AS}. Hence $x$ must therefore also operator commute with $p_F$ for all $F\in\mathcal{F}$ and again since the Jordan multiplication is $\sigma$-weakly continuous,  $x$ operator commutes with $p$. We conclude that $x$ operator commutes with all elements in $M$ by \cite[Lemma~4.2.5]{OS}, proving that $x$ is a central projection.
%If $M_+=C_1\times C_2$ as partially ordered sets, then for $e=(x,y)$ we find that $[0,e]=[0,x]\times[0,y]$ as partially ordered sets. By $(ii)$ we have that $x$ and $y$ are orthogonally complemented central projections. Hence $M=U_xM\oplus U_yM$ and $M_+=U_xM_+\oplus U_yM_+$. Let $z_1\in C_1\setminus\{0\}$. Then $z_1\land z_2=0$ for all $z_2\in C_2$, so $z_1\land y=0$ and so $\|z_1\|^{-1}z_1\in[0,x]$. Thus we have $z_1=U_x\|z_1\|x_1$ for some $x_1\in [0,x]$ and therefore $C_1\subseteq U_xM_+$. Conversely, if $z_1=U_x x_1$ for some $x_1\in M_+$, then $\lambda x_1\land x_2=0$ for all $x_2\in[0,y]$. So, if $z_2\in C_2\setminus\{0\}$, then $\|z_2\|^{-1}(z_1\land z_2)=\|z_2\|^{-1}z_1\land\|z_2\|^{-1}z_2=0$ and hence $z_1\land z_2=0$. We find that $z_1\in C_1$, so $C_1=U_xM_+$. Similarly, it follows that $C_2=U_yM_+$, which proves statement $(ii)$.
\end{proof}

We immediately obtain an analogue of this proposition for arbitrary decompositions.

\begin{corollary}\label{c:arbitrary interval decomposition}
Let $M$ be an atomic JBW-algebra and $I$ an index set. Then the following are equivalent:
\begin{enumerate}
\item $M =\bigoplus_{i \in I} M_i$ as a direct sum of JBW-subalgebras;
\item $M \cong \prod_{i \in I} S_i$ as an internal order product;
\item $M_+ \cong \prod_{i \in I} C_i$ as an internal order product;
\item $[0,e] \cong \prod_{i \in I} [0,x_i]$ as an internal order product for some $x_i \in M$;
\item $[0,e] \cong \prod_{i \in I} [0,z_i]$ as an internal order product for some central projections $z_i \in M$ with supremum $e$.
\end{enumerate}
Moreover, all these types of decompositions are in bijection with each other.
%\begin{itemize}
%\item[$(i)$] If $[0,e]=\prod_{i\in I}[0,x_i]$ as a partially ordered set, then $(x_i)_{i\in I}$ consists of pairwise orthogonal central projections with least upper bound $e$ in $[0,e]$.
%\item[$(ii)$]If $M_+=\prod_{i\in I}C_i$ as a partially ordered set, then there are pairwise orthogonal central projections $(z_i)_{i\in I}$ with least upper bound $e$ in $[0,e]$ such that $C_i=U_{z_i}M_+$ for all $i\in I$.
%\end{itemize}
\end{corollary} 

\begin{proof}
Except for $(iv) \Rightarrow (v)$, the proof of this corollary is the same as the proof of Proposition~\ref{p:interval decomposition}. \\
$(iv) \Rightarrow (v)$: By Remark~\ref{r:sup_in_internal_order_product}, $[0,e] \cong [0,x_i] \times [0, \sup \{x_j \colon j \not= i\}]$, and so it follows from Proposition~\ref{p:interval decomposition} that $x_i$ is a central projection. As before, Remark~\ref{r:sup_in_internal_order_product} implies that $e = \sup_i x_i$.
%Suppose $[0,e]=\prod_{i\in I}[0,x_i]$ as a partially ordered set. Let $\mathcal{F}$ denote the collection of finite subsets of $I$ which is directed by set inclusion and define the increasing net $(x_F)_{F\in\mathcal{F}}$ in $[0,e]$ by 
%\[
%x_F(i):=\begin{cases}x_i&\mbox{if $i\in F$;}\\\\ 0&\mbox{if $i\notin F$.}\end{cases}
%\]
%Then for any $F\in\mathcal{F}$ it follows that $[0,e]=\prod_{i\in F}[0,x_i]\times\left[0,\prod_{i\notin F}x_i\right]$ and  Proposition~\ref{p:interval decomposition} implies that $\{x_i\colon i\in F\}$ consists of pairwise orthogonal central projections. Since $M$ is monotone complete, the net $(x_F)_{F\in\mathcal{F}}$ converges $\sigma$-weakly to its least upper bound $x\in [0,e]$. Moreover, as 
%\[
%e=\sup_{i\in I}x_i=\sup_{i\in I}x_{\{i\}}\le\sup_{F\in\mathcal{F}}x_F = x \leq e 
%\]
%we conclude that $x=e$.
%If $M_+=\prod_{i\in I}C_i$ as a partially ordered set, then for $e=(x_i)_{i\in I}$ we have that $[0,e]=\prod_{i\in I}[0,x_i]$ and by statement $(i)$ we have that $\{x_i\colon i\in I\}$ consists of pairwise orthogonal central projections with least upper bound $e$ that yield the direct sum decomposition $M=\bigoplus_{i\in I}U_{z_i}M$. The argument showing that $C_i=U_{z_i}M_+$ for all $i\in I$ is analogous to the proof for Proposition~\ref{p:cone decomposition}$(ii)$. This proves statement $(ii)$. 
\end{proof}

Since order isomorphisms preserve order products, the following lemma now follows immediately from Corollary~\ref{c:arbitrary interval decomposition}.

\begin{lemma}\label{l:order isom decomp}
Let $M$ and $N$ be atomic JBW-algebras, let $\{z_i\}_{i \in I} \in M$ be a collection of orthogonal central projections with supremum $e_M$ and let $f \colon [0,e_M] \to [0,e_N]$ be an order isomorphism. Then $f$ can be decomposed as $f = \prod_{i\in I} f_i$ where $f_i\colon [0,z_i] \to [0,f(z_i)]$ are the corresponding restrictions which are order isomorphisms for each $i\in I$, and $\{f(z_i)\}_{i \in I}$ are orthogonal central projections with supremum $e_N$.
\end{lemma}

Note that the central atoms of $M$ yield totally ordered parts in the decomposition of the effect algebra and in fact, these are the only ones.

\begin{lemma}\label{l:totally ordered}
Let $M$ be an atomic JBW-algebra and suppose that $[0,e]=[0,z_1]\times[0,z_2]$. Then the order interval $[0,z_1]$ is totally ordered if and only if $z_1$ is a central atom.
\end{lemma}

\begin{proof}
If $z_1$ is a central atom, then by definition the order interval $[0,z_1]$ is totally ordered and order isomorphic to the unit interval $[0,1]$. Conversely, if $[0,z_1]$ is totally ordered, then $z_1$ is a central projection by Proposition~\ref{p:interval decomposition}. But if $z_1 = p + q$ for some non-trivial orthogonal projections $p$ and $q$, then $p$ and $q$ would be two incomparable elements of $[0,z_1]$. Hence $z_1$ must be an atom.
\end{proof}

%Given an order isomorphism $f\colon [0,e_M]\to[0,e_N]$ between the unit order intervals of atomic JBW-algebras, the decomposition $[0,e_M]=\prod_{i\in I}[0,e_{M_i}]$ of the unit order interval is transferred to the unit order interval $[0,e_N]$ by $f$. To see this, define for each $i\in I$ the partially ordered sets 
%\[
%E_i:=\bigl\{f(x_i)\colon x_i\in [0,e_{M_i}]\bigr\}
%\]
%and denote by $f_i$ the restriction of $f$ which maps $[0,e_{M_i}]$ onto $E_i$. Note that $f_i$ is an order isomorphism for all $i\in I$ and that we obtain an order isomorphism $g\colon\prod_{i\in I}E_i\to [0,e_N]$ via 
%\[
%g\bigl((f_i(x_i))_{i\in I}\bigr):=\sup_{i\in I}f_i(x_i).
%\] 
%Hence $g((f_i(x_i))_{i\in I})=\sup_{i\in I}f_i(x_i)=f(\sup_{i\in I}x_i)=f(x)$, so we can identify the unit order interval $[0,e_N]$ with the decomposition $\prod_{i\in I}[0,f(e_{M_i})]$ and $f=\prod_{i\in I}f_i$ for the order isomorphisms $f_i\colon [0,e_{M_i}]\to [0,f(e_{M_i})]$. For future reference, we will summarise this argument in the following lemma.

\section{Order isomorphisms between effect algebras}\label{s:order_isomorphisms}

The conditions for an order isomorphism $f\colon [0,e_M]\to [0,e_N]$ to map the invertible part $(0,e_M]$ onto the invertible part $(0,e_N]$, as well as a complete description for such $f$ are given in this section.
It turns out that the invariance of the invertible parts of the effect algebras is related to the dimension of the centre of the atomic JBW-algebra. 

In the following lemma we introduce an important class of order isomorphisms on effect algebras of JB-algebras that are essential for the description of $f$. This result for operators on a real or complex Hilbert space can be found in \cite[Lemma~3.10]{Semrl} and \cite[Lemma~2.1]{D}. For the reader's convenience we include the analogous proof which holds for general unital JB-algebras. 

\begin{lemma}\label{l:p functions}
Let $A$ be a unital JB-algebra. For any $t<1$ the function $\varphi_t\colon [0,e]\to[0,e]$ defined by 
\[
\varphi_t(x):=x\circ(tx+(1-t)e)^{-1}
\] is an order isomorphism and the collection $\{\varphi_t\colon t<1\}$ forms a group under composition which satisfies $\varphi_t\circ\varphi_s=\varphi_{t+s-ts}$. In particular we have $\varphi_t^{-1}=\varphi_{\frac{t}{t-1}}$.
\end{lemma}

\begin{proof}
The statement trivially holds for $t=0$. So, suppose $t\neq 0$. If $0<t<1$, then we can rewrite $\varphi_t$, by using the identity $y\circ(y+\lambda e)^{-1}=e-\lambda(y+\lambda e)^{-1}$ which holds for any $y\in A_+$ and $\lambda>0$, to be 
\[
\varphi_t(x)={\textstyle \frac{1}{t}}e-{\textstyle \frac{1-t}{t^2}}\left(x+({\textstyle \frac{1}{t}}-1)e\right)^{-1}.
\]
Since we have that $x\le y$ if and only if $y^{-1}\le x^{-1}$ for all $x,y\in A_+^\circ$ by \cite[Lemma~1.31]{AS}, it follows immediately from the alternative description of $\varphi_t$ that it is an order isomorphism. Similarly, if $t<0$, then we can rewrite $\varphi_t$ to be 
\[
\varphi_t(x)={\textstyle \frac{1}{t}}e+{\textstyle \frac{1-t}{t^2}}\left((1-{\textstyle \frac{1}{t}})e-x\right)^{-1}
\]
and it is verified analogously that it is an order isomorphism in this case as well. 

If $t,s<1$, then for $x\in [0,e]$ it follows that

\begin{align*}
\varphi_t(\varphi_s(x))&=\varphi_t(x\circ(sx+(1-s)e)^{-1})\\&=x\circ(sx+(1-s)e)^{-1}\circ\bigl(t(x\circ(sx+(1-s)e))^{-1})+(1-t)e\bigr)^{-1}\\&=x\circ\bigl(tx+s(1-t)x+(1-s)(1-t)e\bigr)^{-1}=x\circ\bigl((t+s-ts)x+(1-(t+s-ts)e)\bigr)^{-1}\\
&=\varphi_{t+s-ts}(x),
\end{align*}
so $\varphi_t\circ\varphi_s=\varphi_{t+s-ts}$. Since $\varphi_0$ is the identity function, we infer from the composition identity that the inverse of $\varphi_t$ satisfies $\varphi_t^{-1}=\varphi_{\frac{t}{t-1}}$. 
%The norm continuity of these functions follows from the fact that quadratic representations are norm continuous by \cite[Theorem~1.25]{AS}, and the identity
%\begin{align*}
%(x+\lambda e)^{-1}-(x_n+\lambda e)^{-1}&= U_{(x+\lambda e)^{-1/2}}(e+U_{(x+\lambda e)^{1/2}}(x_n+\lambda e)^{-1})\\&= U_{(x+\lambda e)^{-1/2}}(e+U_{(x+\lambda e)^{-1/2}}(x_n+\lambda e))
%\end{align*}
%which follows from \eqref{eq:quad properties} and holds for any $\lambda>0$.
\end{proof}

\subsection{Characterising the order isomorphisms}

In unital JB-algebras the order isomorphisms on effect algebras $[0,e]$ are related to order isomorphisms on cones, where the invertible elements $(0,e]$ of this effect algebra play an important role. If $A$ and $B$ are unital JB-algebras and $f\colon [0,e_A]\to [0,e_B]$ is an order isomorphism such that $f$ leaves the invertible elements invariant, that is, the restriction $f|_{(0,e_A]}$ is an order isomorphism onto $(0,e_B]$, then this restriction yields an order isomorphism between the cones $A_+$ and $B_+$. A key observation is that $(0,e_A]$ and $A_+$ are order anti-isomorphic via the map $(0,e_A]\ni x\mapsto x^{-1}-e_A\in A_+$. It then follows that 
\begin{align}\label{e:order isom from open interval}
x\mapsto f|_{(0,e_A]}((x+e_A)^{-1})^{-1}-e_B
\end{align}
is an order isomorphism from $A_+$ onto $B_+$. Here we use the fact that $x\le y$ if and only if $y^{-1}\le x^{-1}$ for all $x,y\in A_+^\circ$ by \cite[Lemma~1.31]{AS}. The order isomorphisms between cones of atomic JBW-algebras are well understood and have been characterised in \cite[Theorem~3.8]{IR}, which we can exploit to find a formula for $f|_{(0,e_A]}$. 

\begin{proposition}\label{p:order isom on interior}
Let $M$ and $N$ be atomic JBW-algebras and suppose that $M$ does not contain any central atoms. If $f \colon (0,e_M] \to (0,e_N]$ is an order isomorphism, then 
\[
f(x)=\left(U_yJx^{-1}-y^2+e_N\right)^{-1}\qquad\bigl(x\in(0,e_M]\bigr)
\] 
for some $y\in N_+^\circ$ and a Jordan isomorphism $J \colon M \to N$.
\end{proposition}

\begin{proof}
Let $f \colon (0,e_M] \to (0,e_N]$ be an order isomorphism. Then the map $\hat{f} \colon M_+ \to N_+$ defined by 
\[
\hat{f}(x):=f((x+e_M)^{-1})^{-1}-e_N
\]
is an order isomorphism. By \cite[Theorem~3.8]{IR}, since $M$ does not contain any disengaged atoms (the disengaged atoms coincide with the central atoms), there is a Jordan isomorphism $J \colon M \to N$ and a $y\in N_+^\circ$ such that $\hat{f}(x) = U_y Jx$ for all $x\in M_+$. Rewriting this yields 
$f((x+e_M)^{-1})=(U_yJx+e_N)^{-1}$ for all $x\in M_+$, so we conclude that $f(x)=(U_yJx^{-1}-y^2+e_N)^{-1}$ for all $x\in (0,e_M]$ as required. 
\end{proof}

\begin{corollary}
Let $A$ be a unital JB-algebra. Suppose that for every $y \in A_+^\circ$ there exists an order automorphism $f_y$ of $[0,e]$ such that
\[
f_y(x)=\left(U_y x^{-1}-y^2+e \right)^{-1}
\]
for all $x \in (0, e]$, then for all $z \in (0,e)$, the map $f_{(z^{-1}-e)^{-1/2}}$ maps $z$ to $\frac{1}{2}e$ and $f_{(z^{-1}-e)^{1/2}}$ maps $\frac{1}{2}e$ to $z$. In particular, the order automorphism group of $[0,e]$ acts transitively on $(0,e)$.
\end{corollary}

The following lemma shows that $(0,e]$ is order dense in $[0,e]$.

\begin{lemma}\label{l:inf dense}
Let $A$ be a JB-algebra and $x\in [0,e]$. Then there exists a monotone decreasing sequence $(x_n)_{n\ge 1}$ in $(0,e]$ with infimum $x$.
\end{lemma}

\begin{proof}
For $n\in\N$, define the continuous functions $f_n$ on $[0,1]$ by
\[
f_n(t):=
\begin{cases}
t & \mbox{if }t\in [\frac{1}{n},1], \\
\frac{1}{n} & \mbox{if }t\in [0,\frac{1}{n}).
\end{cases}
\]
By the continuous functional calculus $x_n := f_n(x)$ defines a monotone decreasing sequence that converges to $x$ in norm and has $x$ as a lower bound. Suppose $y$ is another lower bound for all $x_n$. Then $x-y  = \lim_n x_n - y \geq 0$, showing that $x$ is the greatest lower bound of $(x_n)_{n\ge 1}$.
\end{proof}

\begin{theorem}\label{t:order isom on without central atoms}
Let $M$ and $N$ be atomic JBW-algebras and suppose that $M$ does not contain any central atoms. Then $f \colon [0,e_M] \to [0,e_N]$ is an order isomorphism such that $f((0,e_M])=(0,e_N]$ if and only if $f$ is of the form
\[
f(x)=\varphi_t\left(U_{(z^2+e_N)^{1/2}}(e_N-(e_N+U_{z^{-1}}Jx)^{-1})\right)\qquad\bigl(x\in [0,e_M]\bigr)
\]
for some $t<1$, $z\in N_+^\circ$, and a Jordan isomorphism $J \colon M \to N$.
\end{theorem}

\begin{proof}
Suppose that $f\colon [0,e_M]\to[0,e_N]$ is an order isomorphism such that $f((0,e_M])=(0,e_N]$. It follows from Proposition~\ref{p:order isom on interior} that the restriction of $f$ to $(0,e_M]$ is of the form 

\[
f(x)=(U_yJx^{-1}-y^2+e_N)^{-1}
\]
for some $y\in N_+^\circ$ and a Jordan isomorphism $J\colon M\to N$. Let $\lambda>1$ be such that $\lambda e_N-y^2$ is positive and invertible. Then for $x\in (0,e_M]$, it follows from \eqref{e:op com quad rep} and \eqref{eq:quad properties} that 

\[
U_{(\lambda e_N-y^2)^{-1/2}}U_y=U_{y^{1/2}}U_{(\lambda e_N-y^2)^{-1/2}}U_{y^{1/2}}=U_{{U_{y^{1/2}}}(\lambda e_N-y^2)^{-1/2}},
\]
and so for $z:={U_{y^{1/2}}}(\lambda e_N-y^2)^{-1/2}$ we have 
\[
f(x)^{-1}+(\lambda-1)e_N= U_yJx^{-1}-y^2+\lambda e_N=U_{(\lambda e_N-y^2)^{1/2}}\left(e_N+U_{z} Jx^{-1}\right).
\]
Using the functional calculus on the JB-algebra generated by $y$ and $e_N$ again, we find that 

\[
z=y\circ(\lambda e_N-y^2)^{-1/2}=(y^2\circ(\lambda e_N-y^2)^{-1})^{1/2}=(\lambda(\lambda e_N-y^2)^{-1}-e_N)^{1/2}
\] 
hence $z^2+e_N=\lambda(\lambda e_N-y^2)^{-1}$ and so $\lambda(z^2+e_N)^{-1}=\lambda e_N-y^2$. Therefore,
\begin{align*}
U_{(\lambda e_N-y^2)^{1/2}}\left(e_N+U_{z} Jx^{-1}\right)&=U_{(\lambda(z^2+e_N)^{-1})^{1/2}}\left(e_N+U_zJx^{-1}\right)\\&=U_{\lambda^{1/2}(z^2+e_N)^{-1/2}}\left(e_N+(U_{z^{-1}}Jx)^{-1}\right)\\&=\lambda U_{(z^2+e_N)^{-1/2}}(e_N+(U_{z^{-1}}Jx)^{-1})
\end{align*}
and by using \eqref{eq:quad properties} together with the identity $(e_N+(U_{z^{-1}}Jx)^{-1})^{-1}=e_N-(e_N+U_{z^{-1}}Jx)^{-1}$ to obtain the second equality, we find that 
\begin{align}\label{eq:extend f 1}
\lambda\left(f(x)^{-1}+(\lambda-1)e_N\right)^{-1}&= U_{(z^2+e_N)^{1/2}}(e_N+(U_{z^{-1}}Jx)^{-1})^{-1} \nonumber \\&=U_{(z^2+e_N)^{1/2}}(e_N-(e_N+U_{z^{-1}}Jx)^{-1}).
\end{align}
Note that the identity $\left(e_N+((\lambda-1)f(x))^{-1}\right)^{-1}=e_N-(e_N+(\lambda-1)f(x))^{-1}$ now yields
\begin{align}\label{eq:extend f 2}
\lambda\left(f(x)^{-1}+(\lambda-1)e_N\right)^{-1}&={\textstyle \frac{\lambda}{\lambda-1}}\left(e_N+((\lambda-1)f(x))^{-1}\right)^{-1}={\textstyle \frac{\lambda}{\lambda-1}}\left(e_N-(e_N+(\lambda-1)f(x))^{-1}\right)\nonumber\\&=
{\textstyle \frac{\lambda}{\lambda-1}}e_N-{\textstyle \frac{\lambda}{(\lambda-1)^2}}\left({\textstyle \frac{1}{\lambda-1}}e_N+f(x)\right)^{-1}\nonumber\\&=\varphi_s(f(x))
\end{align}
for $s:=\frac{\lambda-1}{\lambda}\in(0,1)$. Since $\varphi_s^{-1}=\varphi_t$ with $t=\frac{s}{s-1}=1-\lambda\in(-\infty,0)$, we combine \eqref{eq:extend f 1} and \eqref{eq:extend f 2} to write $f(x)$ as
\begin{align}\label{e:expression for f}
f(x)=\varphi_t(\varphi_s(f(x))=\varphi_t\left(U_{(z^2+e_N)^{1/2}}(e_N-(e_N+U_{z^{-1}}Jx)^{-1})\right).
\end{align}
The expression for $f$ in \eqref{e:expression for f} is a well defined order isomorphism between the entire order intervals $[0,e_M]$ and $[0,e_N]$. We will show that the only order isomorphism $g\colon[0,e_M]\to [0,e_N]$ for which the restriction $g|_{(0,e]}$ coincides with $f$ is of the form \eqref{e:expression for f}, proving the required description for $f$. Indeed, if $g\colon [0,e_M]\to [0,e_N]$ is an order isomorphism such that $g|_{(0,e_M]}=f$, and $x\in [0,e_M]$, then by Lemma~\ref{l:inf dense} there is a monotone decreasing sequence $(x_n)_{n\ge 1}$ in $(0,e_M]$ with infimum $x$. Hence
\[
f(x)=f\left(\inf_{n\ge 1}x_n\right)=\inf_{n\ge 1}f(x_n)=\inf_{n\ge 1}g(x_n)=g\left(\inf_{n\ge 1}x_n\right)=g(x)
\]
as $f$ and $g$ are order isomorphisms. We conclude that $f$ must be of the form \eqref{e:expression for f}.

Conversely, suppose that $f$ is of the form \eqref{e:expression for f} for some $t<1$, $z\in N_+^\circ$, and a Jordan isomorphism $J\colon M\to N$. Since $\varphi_s(\eps e_N)=\frac{\eps}{s(\eps-1)+1}e_N\in(0,e_N]$ for all $s<1$ and all $\eps>0$, it follows that $\varphi_s((0,e_N])=(0,e_N]$ for all $s<1$. By \eqref{eq:quad properties} the quadratic representation $U_{(z^2+e_N)^{1/2}}$ maps invertible elements to invertible elements. Lastly, for $x\in[0,e_M]$ it follows from the identity 
\[
e_N-(e_N+U_{z^{-1}}Jx)^{-1}=(U_{z^{-1}}Jx)\circ(e_N+U_{z^{-1}}Jx)^{-1}
\]
in conjunction with the Shirsov-Cohn theorem that $e_N-(e_N+U_{z^{-1}}Jx)^{-1}$ is invertible if and only if $U_{z^{-1}}Jx$ is invertible. This in turn, is equivalent to saying that $x$ is invertible. Hence $f(x)\in(0,e_N]$ if and only if $x\in(0,e_M]$, or equivalently, $f((0,e_M])=(0,e_N]$.
\end{proof}

\subsection{Invariance of $(0,e]$}

In this section we characterise for which atomic JBW-algebras all order isomorphisms satisfy $f((0,e_M])=(0,e_N]$. For $B(H)$ it was shown by \v{S}emrl \cite{Semrl} that every order isomorphism of $[0,e]$ leaves the invertible part $(0,e]$ invariant. For general atomic JBW-algebras this is no longer the case. For example, on the effect algebra $[0,e]\subseteq \ell^\infty$ we construct an order isomorphism $f\colon [0,e]\to [0,e]$ as follows. For $n \geq 1$, let $f_n\colon [0,1]\to[0,1]$ be an order isomorphism such that $f_n(1/2)=2^{-n}$. Then $f(\lambda_n)_{n\ge 1}:=(f_n(\lambda_n))_{n\ge 1}$ is an order isomorphism that maps $\frac{1}{2}e$ to $(2^{-n})_{n\ge 1} \notin (0,e]$. However, it turns out that all order isomorphisms between effect algebras of JBW-factors of type I do leave the invertible parts invariant.

Before proceeding, we will introduce some terminology and notation. For a JBW-algebra $M$, we denote by $\P_1(M)$ the set of atoms of $M$. For an atom $p$ in a JBW-algebra $M$ and $x\in M_+$, we say that $x$ \emph{dominates the atom} $p$ if there exists a $\lambda>0$ such that $\lambda p\le x$. 

\begin{proposition}\label{p:(0,e] invariant for factors}
Let $M$ and $N$ be JBW-factors of type $I$ and let $f\colon [0,e_M]\to[0,e_N]$ be an order isomorphism. Then $f(0,e_M]=(0,e_N]$. 
\end{proposition}

\begin{proof}
By \cite[Lemma~3.1]{IR}, the rays generated by the atoms correspond exactly to extreme rays of the cone, and so the maximal totally ordered subsets $S$ of $[0,e_M]$ such that $[0,x] \subseteq S$ for all $x \in S$ are precisely of the form $[0,p]$ for $p \in \P_1(M)$. Hence, if $p \in \P_1(M)$, then $f([0,p]) = [0,q]$ for some $q \in \P_1(N)$. In particular, $f$ maps $\P_1(M)$ bijectively onto $\P_1(N)$. 

Now if $x \in (0,e_M]$, then $x$ dominates every atom in $M$ and hence by the above $f(x)$ dominates every atom in $N$. Suppose $N$ is of finite rank. Then the spectral decomposition yields that 
\[
f(x) = \sum_{i=1}^n \lambda_i q_i
\]
for some $\lambda_i \in [0,1]$ and $q_i \in \P_1(N)$. Since $f(x)$ dominates every atom in $N$, all the $\lambda_i$ must be strictly positive, and so $0 \notin \sigma(f(x))$.

It remains to consider the case where $N$ is of infinite rank. Then by \cite[Theorem~7.5.11]{OS} $N$ is of the form $B(H)_{sa}$ where $H$ is a real, complex or quaternion Hilbert space. Assume $H$ is complex. The atoms are precisely the rank one projections, and since $f(x)$ dominates every atom, by Douglas' lemma (\cite[Theorem~1]{Douglas}) it follows that the range of $f(x)$ contains the range of every rank 1 projection. Hence $f(x)$ is surjective, and $f(x)$ has to be injective as well otherwise $f(x)$ would not dominate a rank 1 projection in its kernel. Hence $f(x)$ is invertible.

As explained in \cite{Douglas}, Douglas' lemma is also valid for real Hilbert spaces, so it remains to verify Douglas' lemma for quaternion Hilbert spaces. An inspection of Douglas' proof shows that it only relies on the Closed Graph Theorem, which also holds in quaternion Hilbert spaces by \cite[Theorem~3.9]{ACS} (an earlier but more archaic reference is \cite[Corollaire~5,~p.~I.19]{Bourbaki}). 

An alternative proof that avoids quaternion Hilbert spaces is as follows. As explained in \cite[7.5.5,~7.5.6,~7.5.10,~7.5.11]{OS} one can model the bounded self-adjoint operators on a quaternion Hilbert spaces as follows: let $H$ be a complex Hilbert space and let $j \colon H^2 \to H^2$ be the unit quaternion defined by $j(\xi, \eta) := (-\overline{\eta}, \overline{\xi})$, where $\xi \mapsto \overline{\xi}$ is a conjugation on $H$. Then $N$, the bounded self-adjoint operators on a quaternion Hilbert space, can be represented as the bounded self-adjoint operators on $H^2$ that commute with $j$. For $\xi \in H$, let $P_\xi$ denote the projection onto the span of $\xi$; then one readily verifies that the self-adjoint operator $T := P_\xi \oplus P_{\overline{\xi}}$ commutes with $j$ and hence belongs to $N$. Since $f(x)$ dominates every atom and $T$ has at most rank 2 in $N$, it also dominates $T$, and so it will also dominate $P_\xi \oplus 0$ and $0 \oplus P_{\overline{\xi}}$ for every $\xi \in H$; by taking $\eta := \overline{\xi}$, $T$ also dominates $0 \oplus P_\eta$ for every $\eta \in H$. By Douglas' lemma for complex Hilbert space operators, the range of $f(x)$ contains the ranges of $P_\xi \oplus 0$ and $0 \oplus P_\eta$, and so the range of $f(x)$ contains all vectors $(\xi, \eta) \in H^2$. Hence $f(x)$ is surjective which as before shows that $f(x)$ is invertible.
\end{proof}

Note that a consequence of Proposition~\ref{p:(0,e] invariant for factors} is that any order isomorphism between effect algebras of atomic JBW-algebras that are a finite direct sum of type I factors must preserve the invertible parts as well. In fact, the following lemma shows that this characterises when it happens. 
%Note that if the unit order interval $[0,e_M]$ of an atomic JBW-algebra $M$ admits a finite maximal decomposition 
%\[
%[0,e_M]=\prod_{k=1}^n[0,e_{M_k}], 
%\]
%then $\{e_{M_1},\ldots,e_{M_n}\}$ consists of pairwise orthogonal central projections by Corollary~\ref{c:arbitrary interval decomposition} and for an atomic JBW-algebra $N$ with unit order interval $[0,e_N]$, an order isomorphism $f\colon [0,e_M] \to [0,e_N]$ would yield a maximal decomposition of the unit order interval 
%\[
%[0,e_N]=\prod_{k=1}^n[0,f(e_{M_k})]. 
%\]
%Hence $\{f(e_{M_1}),\ldots,f(e_{M_n})\}$ is a set of pairwise orthogonal central projections and $[0,f(e_{M_k})]$ are unit order intervals of type I factors of $N$. By Proposition~\ref{p:(0,e] invariant for factors} the restrictions of $f$ mapping $[0,e_{M_k}]$ onto $[0,f(e_{M_k})]$ must satisfy $f(0,e_{M_k}]=(0,f(e_{M_k})]$. Consequently, the finite decomposition of $[0,e_M]$ yields that $f(0,e_M]=(0,e_N]$. 

\begin{lemma}\label{l:finite dimensional centre}
Let $M$ and $N$ be atomic JBW-algebras such that $[0, e_M]$ is order isomorphic to $[0, e_N]$. Then any order isomorphism $f\colon [0,e_M]\to [0,e_N]$ satisfies $f(0,e_M]=(0,e_N]$ if and only if the centre $Z(M)$ is finite dimensional.
\end{lemma}

\begin{proof}
Suppose that all order isomorphisms $f\colon [0,e_M]\to [0,e_N]$ satisfy $f(0,e_M]=(0,e_N]$ and that $Z(M)$ is infinite dimensional, then $Z(M)$ must contain a countable set of pairwise orthogonal central projections $(p_n)_{n\ge 1}$, so we can decompose $M$ as a direct sum 
\[
M=\bigoplus_{n=1}^\infty U_{p_n}(M)\oplus M'. 
\]
The order interval $[0,e]$ can now be written as the product $[0,e]=\prod_{n=1}^\infty[0,p_n]\times [0,e_{M'}]$. For $n\in\N$ let $t_n:=\frac{1}{2}(3-2^n)$, then $t_n<1$ and by Lemma~\ref{l:p functions} we have order isomorphisms $\varphi_{t_n}\colon [0,p_n]\to [0,p_n]$ such that $\varphi_{t_n}(\frac{1}{2}p_n)=2^{-n}p_n$ for $n\in\N$. Together with the identity function $\mathrm{Id}_{[0,e_{M'}]}$ on $[0,e_{M'}]$ by Lemma~\ref{l:order isom decomp} we obtain an order isomorphism 
\[
g:=\prod_{n=1}^\infty\varphi_{t_n}\times\mathrm{Id}_{[0,e_{M'}]}\colon[0,e_M]\to [0,e_M]. 
\]
For any order isomorphism $f\colon[0,e_M]\to[0,e_N]$ it follows that $f\circ g\colon [0,e_M]\to[0,e_N]$ is an order isomorphism as well, but $(f\circ g)(\frac{1}{2}e_M)\notin(0,e_N]$ as $0\in\sigma(g(\frac{1}{2}e_M))$, so $g(\frac{1}{2}e_M)\notin (0,e_M]$ and $f([0,e_M]\setminus(0,e_M])=[0,e_N]\setminus(0,e_N]$. 

Conversely, suppose that $Z(M)$ is finite dimensional and let $f\colon [0,e_M]\to[0,e_N]$ be an order isomorphism. Then $M$ is a finite direct sum $M=M_1\oplus\cdots\oplus M_n$ of type $I$ factors which yields the decompositions $[0,e_M]=\prod_{k=1}^n[0,e_{M_k}]$ and $(0,e_M]=\prod_{k=1}^n(0,e_{M_k}]$ as partially ordered sets. By Lemma~\ref{l:order isom decomp} the image of the order isomorphism decomposes $[0,e_N]$ as the partially ordered set 
\[
[0,e_N]=\prod_{k=1}^n f[0,e_{M_k}]. 
\]
As $e_{M_k}$ is a central projection that cannot be decomposed further as a non-trivial sum of central projections, it follows from Corollary~\ref{c:arbitrary interval decomposition} that $f[0,e_{M_k}]=[0,e_{N_l}]$ where $e_{N_l}$ is a central projection which cannot be decomposed any further as a non-trivial sum of central projections either. Hence $U_{e_{N_l}}(N)$ is a factor and  $f(0,e_{M_k}]=(0,e_{N_l}]$ by Proposition~\ref{p:(0,e] invariant for factors}. We conclude that 
\[
f(0,e_M]=f\left(\prod_{k=1}^n(0,e_{M_k}]\right)=\prod_{k=1}^nf(0,e_{M_k}]=\prod_{l=1}^n(0,e_{N_l}]=(0,e_N]
\]
as required.
\end{proof}

\subsection{Main results}

For a general atomic JBW-algebra $M$, consider the collection $\mathcal{P}_1(M)\cap Z(M)$ of central atoms and we define 
\[
p:=\sup\{q\colon q\in \mathcal{P}_1(M)\cap Z(M)\} 
\]
in the lattice of projections. Since Jordan multiplication is separately $\sigma$-weakly continuous, it follows that $p$ operator commutes with all projections, which in turn implies that $p$ must be central by \cite[Lemma~4.2.5]{OS}. Hence the central projection $p$ decomposes $M$ as an algebraic direct sum $M=U_p(M)\oplus U_{p^\perp}(M)$ where $U_{p^\perp}(M)$ does not contain any central atoms. Following the terminology used in \cite{IR}, we call $U_p(M)$ the \emph{disengaged part} of $M$ and $U_{p^\perp}(M)$ the \emph{engaged part} of $M$. Note that the disengaged part of $M$ is precisely the associative part, i.e., the type $I_1$ part. Furthermore, we will denote by $M_D$ the disengaged part of $M$ and by $M_E$ the engaged part of $M$. In order to abbreviate the notation and to incorporate the atoms in $M_D$, we write $\mathcal{D}_M:=\mathcal{P}_1(M)\cap Z(M)$ and it follows from \cite[Proposition~3.5]{IR} that 
\[
M_D=\bigoplus_{p\in\mathcal{D}_M}U_p(M)=\bigoplus_{p\in\mathcal{D}_M}\mathbb{R}p=\ell_\infty(\mathcal{D}_M,\mathbb{R}),
\] 
as $U_p(M)$ is one-dimensional by \cite[Lemma~3.29]{AS}. Using the disengaged and engaged parts of $M$, we can decompose the effect algebra $[0,e_M]$ as
\begin{align}\label{e:disengaged/engaged parts of interval}
[0,e_M]=[0,e_{M_D}]\times[0,e_{M_E}]=[0,1]^{\mathcal{D}_M}\times[0,e_{M_E}]
\end{align}
and give the following description of the order isomorphisms between the effect algebras of atomic JBW-algebras. 

\begin{theorem}\label{t:char order isoms finite sums}
Let $M$ and $N$ be atomic JBW-algebras and let $f\colon [0,e_M]\to[0,e_N]$ be an order isomorphism. For $x \in M = M_D \oplus M_E$ we write $x=(x_p)_{p\in\mathcal{D}_M}\times x_E$ as in \eqref{e:disengaged/engaged parts of interval}.
\begin{enumerate}
\item There exists an index set $I$ such that $M_E$ decomposes as a direct sum of factors $\bigoplus_{i \in I} M_i$ and $N_E$ decomposes as a direct sum of factors $\bigoplus_{i \in I} N_i$, and for $x_E = (x_i)_{i \in I}$,
\[
f(x)=(f_p(x_p))_{\sigma(p)\in\mathcal{D}_N}\times \prod_{i \in I} \varphi_{t_i} \left( U_{(z_i^2+e_{N_i})^{1/2}}(e_{N_i}-(e_{N_i}+U_{z_i^{-1}}J_i x_i)^{-1})\right)
\]
for some $t_i < 1$, $z_i \in (N_i)_+^\circ$, a bijection $\sigma\colon \mathcal{D}_M\to\mathcal{D}_N$, order isomorphisms $f_p \colon [0,1] \to [0,1]$ for all $p \in\mathcal{D}_M$, and Jordan isomorphisms $J_i \colon M_i \to N_i$. 
\item If $M_E$ is a finite direct sum of factors, then
\[ 
f(x)=(f_p(x_p))_{\sigma(p)\in\mathcal{D}_N}\times \varphi_t\left(U_{(z^2+e_{N_E})^{1/2}}(e_{N_E}-(e_{N_E}+U_{z^{-1}}Jx_E)^{-1})\right)
\]
for some $t<1$, $z\in (N_E)_+^\circ$, a bijection $\sigma\colon \mathcal{D}_M\to\mathcal{D}_N$, order isomorphisms $f_p\colon [0,1]\to [0,1]$ for all $p\in\mathcal{D}_M$, and some Jordan isomorphism $J\colon M_E\to N_E$. 
\item If $M_E$ is a finite direct sum of factors and $M_D = \{0\}$, then
\[ 
f(x) = \varphi_t\left(U_{(z^2+e_N)^{1/2}}(e_N - (e_N + U_{z^{-1}}Jx)^{-1})\right)
\] 
for some $t<1$, $z\in N_+^\circ$, and some Jordan isomorphism $J\colon M \to N$. 
\end{enumerate}

\end{theorem}

\begin{proof}
$(i)$: By \eqref{e:disengaged/engaged parts of interval} we can write 
\[
[0,e_M]=[0,1]^{\mathcal{D}_M}\times[0,e_{M_E}]\qquad\mbox{and}\qquad [0,e_N]=[0,1]^{\mathcal{D}_N}\times[0,e_{N_E}].
\]
Since the order isomorphism $f$ must preserve the totally ordered parts of the decomposition, it follows that $f=f_D\times f_E$ where $f_D\colon [0,1]^{\mathcal{D}_M}\to[0,1]^{\mathcal{D}_N}$ and $f_E\colon [0,e_{M_E}]\to[0,e_{N_E}]$ are the corresponding order isomorphism restrictions of $f$ by Lemma~\ref{l:order isom decomp} and Lemma~\ref{l:totally ordered}. Moreover, $f_D$ induces a bijection $\sigma\colon \mathcal{D}_M\to \mathcal{D}_N$ so that $f_D((x_p)_{p\in\mathcal{D}_M})=(f_p(x_p))_{\sigma(p)\in\mathcal{D}_N}$ where $f_p\colon[0,1]\to[0,1]$ is an order isomorphism for each $p\in\mathcal{D}_M$. 

The atomic JBW-algebra $M_E$ is a direct sum $\bigoplus_{i\in I}M_i$ of type I factors by \cite[Proposition~3.45]{AS} and so the corresponding effect algebra is of the form $[0,e_{M_E}]=\prod_{i\in I}[0,e_{M_i}]$. Hence $[0,e_{N_E}]$ is of the form $\prod_{i\in I}[0,f(e_{M_i})]$ and $f_E=\prod_{i\in I} f_i$ where $f_i\colon[0,e_{M_i}]\to[0,f(e_{M_i})]$ are the restriction order isomorphisms by Lemma~\ref{l:order isom decomp}. It follows from Proposition~\ref{p:interval decomposition} that the effect algebra $[0,f(e_{M_i})]$ belongs to a factor $N_i$ in $N_E$ for each $i\in I$ and $N_E=\bigoplus_{i\in I}N_i$. 

Since $M_i$ does not contain any central atoms and $f_i((0,e_{M_i}])=(0,e_{N_i}]$ by Lemma~\ref{l:finite dimensional centre}, it follows from Theorem~\ref{t:order isom on without central atoms} that there is a $t_i<1$, an element $z_i \in (N_i)_+^\circ$, and a Jordan isomorphism $J_i \colon M_i\to N_i$ such that 
\[
f_i(x_i) = \varphi_{t_i} \left( U_{(z_i^2+e_{N_i})^{1/2}}(e_{N_i}-(e_{N_i}+U_{z_i^{-1}}J_i x_i)^{-1})\right)
\]
which yields the required description of the order isomorphism $f$.

$(ii)$: The proof is similar to $(i)$, except now $f_E(0, e_M] = (0, e_N]$ by Lemma~\ref{l:finite dimensional centre} and so we can use the same arguments for $f_E$ that characterise $f_i$.

$(iii)$: Follows immediately from $(ii)$.
\end{proof}
Note that the characterisation of order isomorphisms in Theorem~\ref{t:char order isoms finite sums} and the order isomorphisms on the unit interval $[0,1]$ determine all order isomorphisms between effect algebras of atomic JBW-algebras. 

\begin{corollary}
Let $M$ and $N$ be atomic JBW-algebras. Then $[0,e_M]$ and $[0,e_N]$ are order isomorphic if and only if $M$ and $N$ are Jordan isomorphic. 
\end{corollary}

\begin{proof}
If $[0,e_M]$ and $[0,e_N]$ are order isomorphic, then by decomposing $[0,e_M]=[0,e_{M_D}]\times[0,e_{M_E}]$ and $[0,e_N]=[0,e_{N_D}]\times[0,e_{N_E}]$, an order isomorphism $f\colon [0,e_M]\to [0,e_N]$ can be written as $f=f_D\times f_E$ where $f_D\colon [0,e_{M_D}]\to[0,e_{N_D}]$ and $f_E\colon [0,e_{M_E}]\to [0,e_{N_E}]$ are the corresponding order isomorphism restrictions of $f$ by Lemma~\ref{l:order isom decomp} and Lemma~\ref{l:totally ordered}. Moreover, $f_D$ induces a bijection $\sigma\colon\mathcal{D}_M\to\mathcal{D}_N$ which can be used to define a Jordan isomorphism $J_D\colon M_D\to N_D$ via 
\[
J_D((x_p)_{p\in\mathcal{D}_M}):=(x_p)_{\sigma(p)\in\mathcal{D}_N}.
\]
By Theorem~\ref{t:char order isoms finite sums}$(i)$ there is a Jordan isomorphism $J_i\colon M_i\to N_i$ for each $i\in I$ and we can define a Jordan isomorphism $J_E:=\bigoplus_{i\in I}J_i\colon M_E\to N_E$ by $J_E((x_i)_{i\in I}):=(J_i(x_i))_{i\in I}$. We conclude that $J:=J_D\oplus J_E\colon M\to N$ is a Jordan isomorphism and so $M$ and $N$ are Jordan isomorphic. 

Conversely, every Jordan isomorphism $J \colon M \to N$ restricts to an order isomorphism between $[0,e_M]$ and $[0, e_N]$.
\end{proof}

\footnotesize
\bibliographystyle{amsalpha}
\bibliography{effect-algebra-bib}

\providecommand{\bysame}{\leavevmode\hbox to3em{\hrulefill}\thinspace}
\providecommand{\MR}{\relax\ifhmode\unskip\space\fi MR }
\providecommand{\MRhref}[2]{%
  \href{http://www.ams.org/mathscinet-getitem?mr=#1}{#2}
}
\providecommand{\href}[2]{#2}
\begin{thebibliography}{HOS84}

\bibitem[ACS15]{ACS}
Daniel Alpay, Fabrizio Colombo, and Irene Sabadini, \emph{Inner product spaces
  and {K}rein spaces in the quaternionic setting}, Recent advances in inverse
  scattering, {S}chur analysis and stochastic processes, Oper. Theory Adv.
  Appl., vol. 244, Birkh\"{a}user/Springer, Cham, 2015, pp.~33--65.

\bibitem[AS03]{AS}
Erik~M. Alfsen and Frederic~W. Shultz, \emph{Geometry of state spaces of
  operator algebras}, Mathematics: Theory \& Applications, Birkh\"{a}user
  Boston, Inc., Boston, MA, 2003.

\bibitem[Bou81]{Bourbaki}
Nicolas Bourbaki, \emph{Espaces vectoriels topologiques. {C}hapitres 1 \`a 5},
  new ed., Masson, Paris, 1981, \'{E}l\'{e}ments de math\'{e}matique. [Elements
  of mathematics].

\bibitem[Con90]{conway}
John~B. Conway, \emph{A course in functional analysis}, second ed., Graduate
  Texts in Mathematics, vol.~96, Springer-Verlag, New York, 1990.

\bibitem[CS50]{CS}
Claude Chevalley and R.~D. Schafer, \emph{The exceptional simple {L}ie algebras
  {$F_4$} and {$E_6$}}, Proc. Nat. Acad. Sci. U.S.A. \textbf{36} (1950),
  137--141.

\bibitem[Dou66]{Douglas}
Ronald~G. Douglas, \emph{On majorization, factorization, and range inclusion of
  operators on {H}ilbert space}, Proc. Amer. Math. Soc. \textbf{17} (1966),
  413--415.

\bibitem[Drn18]{D}
Roman Drnov\v{s}ek, \emph{On order automorphisms of the effect algebra}, Acta
  Sci. Math. (Szeged) \textbf{84} (2018), no.~3-4, 431--437.

\bibitem[HOS84]{OS}
Harald Hanche-Olsen and Erling St{\o}rmer, \emph{Jordan operator algebras},
  Monographs and Studies in Mathematics, vol.~21, Pitman (Advanced Publishing
  Program), Boston, MA, 1984.

\bibitem[IRP95]{jbisometries}
Jos\'{e}~M. Isidro and \'{A}ngel Rodr\'{\i}guez-Palacios, \emph{Isometries of
  {${\rm JB}$}-algebras}, Manuscripta Math. \textbf{86} (1995), no.~3,
  337--348.

\bibitem[Lud83]{Ludwig}
G\"{u}nther Ludwig, \emph{Foundations of quantum mechanics. {I}}, Texts and
  Monographs in Physics, Springer-Verlag, New York, 1983, Translated from
  German by Carl A. Hein.

\bibitem[Mol01]{Molnar2001}
Lajos Moln\'{a}r, \emph{Order-automorphisms of the set of bounded observables},
  J. Math. Phys. \textbf{42} (2001), no.~12, 5904--5909.

\bibitem[Mol03]{Molnar2003}
\bysame, \emph{Preservers on {H}ilbert space effects}, Linear Algebra Appl.
  \textbf{370} (2003), 287--300.

\bibitem[Mor19]{Mori}
Michiya Mori, \emph{Order isomorphisms of operator intervals in von {N}eumann
  algebras}, Integral Equations Operator Theory \textbf{91} (2019), no.~2, Art.
  11, 26.

\bibitem[{\v{S}}em17]{Semrl}
Peter {\v{S}}emrl, \emph{Order isomorphisms of operator intervals}, Integral
  Equations Operator Theory \textbf{89} (2017), no.~1, 1--42.

\bibitem[Tak02]{Takesaki}
Masamichi Takesaki, \emph{Theory of operator algebras. {I}}, Encyclopaedia of
  Mathematical Sciences, vol. 124, Springer-Verlag, Berlin, 2002, Reprint of
  the first (1979) edition, Operator Algebras and Non-commutative Geometry, 5.

\bibitem[vIR20]{IR}
Hendrik van Imhoff and Mark Roelands, \emph{Order isomorphisms between cones of
  {JB}-algebras}, Studia Math. \textbf{254} (2020), no.~2, 179--198.

\end{thebibliography}

\end{document}